\documentclass[12pt]{article}      
\usepackage{nicematrix,fourier,extsizes,blkarray,empheq}
\usepackage{array}
\usepackage{tikz-cd}
\newcolumntype{C}{>{$}c<{$}}
\usepackage{blkarray}
\usepackage{relsize,bm}

\usepackage{amssymb}
\usepackage{eucal}
\usepackage{amsmath}
\usepackage{amsthm}
\usepackage{graphicx}
\usepackage{float}
\usepackage{framed}
\usepackage{enumerate}
\usepackage{hyperref}
\usepackage[toc,page]{appendix}

\newcommand{\R}{{\mathbb  R}}   
\newcommand{\sm}{\mathlarger{\mathlarger{{\Sigma}}}}
 
\newcounter{mastercounter}
\numberwithin{mastercounter}{section}

\newtheorem{thm}[mastercounter]{Theorem}
\newtheorem{lemma}[mastercounter]{Lemma}

\newtheorem{rem}[mastercounter]{Remark}
\makeatletter
\let\c@equation\c@mastercounter
\makeatother

\numberwithin{equation}{section}

\usepackage[T1]{fontenc}

\begin{document}

\title{
Optimization with inequality constraints by the embedded gradient vector field method
}

\author{Petre Birtea, Ioan Ca\c su, Dan Com\u{a}nescu\\
{\small Department of Mathematics, West University of Timi\c soara} 
\\
{\small Bd. V. P\^ arvan, No 4, 300223 Timi\c soara, Rom\^ania}\\
{\small Email: petre.birtea@e-uvt.ro, ioan.casu@e-uvt.ro, dan.comanescu@e-uvt.ro}}
\date{}

\maketitle

\begin{abstract}
We develop a geometric framework for constrained optimization problems with inequality constraints through the introduction of quadratic slack variables. This formulation makes it possible to employ the language of Riemannian geometry and to solve the problem via the embedded gradient vector field method. We lift the feasible set to a smooth submanifold of an extended ambient space. The stratified structure of the resulting constraint manifold is analyzed in detail, yielding a natural partition according to which constraints are active. Using the embedded gradient vector field formalism, we derive explicit, determinantal formulas for the Lagrange multiplier functions directly from the geometry of the constraint manifold, recovering and re-framing the classical Karush-Kuhn-Tucker first-order necessary conditions without invoking the classical Lagrange multiplier method. Second-order optimality conditions are obtained by computing the restricted Hessian on each stratum, and a complete sign condition on the Lagrange multipliers is identified as the geometric counterpart of the classical complementary slackness condition. The theory is illustrated on the double ice-cream cone example, where the geometry of the problem determines the nature and number of local minima.

\end{abstract}

\maketitle

\section{Introduction}

For smooth functions $f_0, f_i, h_j:\R^N\to \R$, we consider the following optimization problem\footnote{The argmax problem is equivalent with the argmin problem for the objective function $-f_0$.}
\begin{equation}\label{with-equalities-121}
\begin{aligned}
\underset{{\bf x}\in \R^N}{\arg\min}\quad & f_0(\mathbf{x}) \\
\text{s.t.}\quad & f_i(\mathbf{x}) \le 0, \quad i = 1,\dots,m \\
 & h_j({\bf x})=0, \quad   j=1,\dots, p.
\end{aligned}
\end{equation}

Under appropriate regularity and constraint qualification assumptions on the objective function $f_0$  and the constraint functions $f_i$ and $h_j$, the optimization problem can be characterized by the Karush-Kuhn-Tucker (KKT) conditions. 
The Karush-Kuhn-Tucker conditions occupy a central role in nonlinear optimization. Originally derived independently by Karush (1939) and by Kuhn and Tucker (1951), they provide necessary (and, under a convexity assumption, sufficient) first-order optimality conditions for problems with inequality and equality constraints. Standard treatments appear in the monographs of Rockafellar \cite{rockafellar}, Bertsekas \cite{bertsekas}, and Nocedal and Wright \cite{nocedal}.
These conditions extend the classical method of Lagrange multipliers to problems with inequality constraints. 

A classical device for converting inequality constraints into equality constraints is the introduction of the so-called slack variables. The squared slack variable technique, in the form $f_i(x) +\frac{1}{2}z_i^2= 0$, appears in the context of Newton-type methods for general mathematical programming in the work of Tapia \cite{tapia}. A detailed modern analysis of this technique, including its implications for regularity, constraint qualifications, and the relationship between the original and slack-extended problems, is given by Fukuda and Fukushima \cite{fukushima}, who clarify several subtle points regarding the equivalence of KKT points and second-order conditions before and after the slack transformation.

Our purpose is to solve this optimization problem by means of the embedded gradient  vector field method. This method was developed by the authors in \cite{birtea-comanescu}, originally in the context of geometric dissipation for dynamical systems, and extended to second-order (Hessian) analysis on constraint manifolds in \cite{Birtea-Comanescu-Hessian}. These works place constrained optimization firmly within Riemannian geometry: the gradient and Hessian of a restricted function are expressed intrinsically in terms of determinantal formulas involving Gram matrices of constraint gradients, bypassing the need to explicitly solve for Lagrange multipliers.

The present paper develops this programme in the setting of inequality constraints.
The central idea is to introduce quadratic slack variables $z_i$, converting each
inequality $f_i(x) \leq 0$ into an equality $f_i(x) + \dfrac{1}{2}z_i^2 = 0$, and to
regard the resulting level set as a submanifold of an extended ambient space
$\mathcal{M} \times \mathbb{R}^m$. Away from singular points, those at which the gradients
of the slack constraint functions become linearly dependent, the Regular Value
Theorem endows this level set with a smooth manifold structure, and the embedded
gradient vector field formalism applies directly.

A key feature of this approach is the natural stratification it induces on the feasible
set. The constraint manifold decomposes into strata indexed by subsets of active
constraints: a fully interior stratum ${\bf S}_0$, where all slack variables are nonzero and
no constraint is active, and boundary strata ${\bf S}_{(i_1, \ldots, i_k)}$, where precisely the
constraints $f_{i_1}, \ldots, f_{i_k}$ are active.  We prove that each stratum, and its projection
onto the original variable space $\mathcal{M}$, is itself a smooth submanifold. This stratified
structure organizes the search for local minima into a finite collection of
subproblems, one for each possible active set, and is analyzed in detail in
Section~2. 

The stratification of the feasible set by active constraints, exploited systematically in the present paper, is related to the Whitney stratification theory of semi-algebraic sets; see \cite{goresky} and the optimization-oriented treatment in \cite{shapiro}. The idea of organizing optimality conditions stratum by stratum, distinguishing the interior, boundary faces, and their intersections, connects also to the theory of active-set methods in nonlinear programming, see \cite{nocedal} and to the study of partly smooth functions introduced in \cite{lewis}, where manifold-like structure of the active constraint set is similarly exploited to obtain sharper second-order conditions. 

Within each stratum, the embedded gradient vector field method yields explicit
determinantal formulas for the Lagrange multiplier functions $\mu_i$ associated with
the active constraints, expressed directly in terms of Gram matrices of constraint
gradients via Cramer's rule. The vanishing of the embedded gradient is shown to
be equivalent to the classical KKT stationarity condition, while the complementary
slackness condition emerges geometrically from the product structure of the tangent
space to the constraint manifold. Crucially, the sign requirement $\mu_i \leq 0$ for a
minimum appears not as an additional hypothesis imposed from outside, but as a
necessary consequence of the positive semi-definiteness of the restricted Hessian on
the constraint manifold. Second-order sufficient conditions are likewise obtained by
computing the restricted Hessian on each stratum, yielding a complete set of first-
and second-order optimality tests summarized in Figure~1.
Second-order necessary and sufficient conditions for nonlinear programs have a long history; see the surveys \cite{bonnans} and the treatment in\cite{rusz}. The sign conditions on Lagrange multipliers identified in the present paper, strict negativity for a strict local minimum, correspond to the classical strict complementarity condition, whose role in sensitivity analysis and in the convergence of interior-point methods is well documented in \cite{wright}.

In Section~3, we illustrate the theory on the double ice-cream cone problem, which seeks
to minimize the linear function $f_0({\bf x}): = c_1 x_1 + c_2 x_2 + c_3 x_3$ over the
intersection of a ball and a solid cone in $\mathbb{R}^3$. This example is chosen
because the geometry of the feasible set, and in particular the position of the
vector ${\bf c}$ relative to the cone, determines both the number and nature of the local
minima in a way that is transparent within the stratified framework. When ${\bf c}$ lies
strictly inside the cone the unique local minimum is found on the spherical-cap
stratum; when ${\bf c}$ lies strictly outside the cone two local minima appear on the
stratum where both constraints are simultaneously active; and when ${\bf c}$ lies exactly
on the cone the first- and second-order tests are inconclusive on part of the boundary
stratum, and an ad hoc argument using compactness identifies the global/local minimizers.
This case-by-case analysis demonstrates both the power and the current limitations
of the method, and points naturally toward future work on degenerate cases.

\section{Necessary optimality and sufficient optimality conditions for problems with inequality constraints}

The first stage is to organize the equality constraints as a Riemannian manifold, i.e. suppose that $\mathcal{M}:=\{{\bf x}\in \R^N\,|\,h_j({\bf x})=0, \,\,j\in \{1,\dots,p\}\}$ is a smooth manifold of dimension $N-p$, a condition guaranteed when the null element ${\bf 0}_p\in \R^p$ is  a regular value for the set of equality constraints. We endow this manifold with the induced Riemannian metric, denoted by ${\bf g}$, from the ambient space $\R^N$. The initial problem becomes
\begin{equation}\label{without-equlities-998} 
\begin{aligned}
\underset{\mathbf{x} \in \mathcal{M}}{\arg\min}\quad & f_0(\mathbf{x}) \\
\text{s.t.}\quad & f_i(\mathbf{x}) \le 0, \quad i = 1,\dots,m
\end{aligned}
\end{equation}

The second stage is to introduce a set of so-called ``slack variables''  that further transform the inequality constraints into equality constraints. Let $F_i:\mathcal{M}\times \R^m\to \R$ be the smooth functions, that we call {\bf the slack constraint functions}, $F_i({\bf x}, {\bf z}):=f_i({\bf x})+\frac{1}{2}z_i^2$. It is immediate to see that the feasible set can be described by the set equality 
$$\{{\bf x}\in \mathcal{M}\,|\,f_i({\bf x})\leq 0,\,i\in \{1,\dots,m\}\}=\Pi ({\bf F}^{-1}({\bf 0}_{m})),$$
where ${\bf F}=(F_1,\dots,F_m)$ and
 $$\Pi:\mathcal{M}\times \R^m\to \mathcal{M},\,\,\Pi ({\bf x},{\bf z}):={\bf x}$$ 
 is the projection on $\mathcal{M}$ along $\R^m$.

Observing the above set equality, we relate the initial optimization problem \eqref{with-equalities-121}, or equivalently \eqref{without-equlities-998}, with the following optimization problem with equality constraints, 
\begin{equation}\label{new-problem-554} 
\begin{aligned}
\underset{({\bf x}, {\bf z})\in {\bf F}^{-1}({\bf 0}_{m})}{\arg\min}\quad & F_0(\mathbf{x}, {\bf z}) 
\end{aligned}
\end{equation}
where $F_0:\mathcal{M}\times \R^m\to \R$ be the smooth function $F_0({\bf x}, {\bf z}):=f_0({\bf x})$. 

The next result appears in \cite{tapia} (without proof), and is also cited in \cite{fukushima}.

\begin{thm}\label{echivalenta777}
The optimization problems \eqref{without-equlities-998} and  \eqref{new-problem-554} are equivalent. More precisely:
\begin{enumerate}[(i)]
\item If ${\bf x}^*\in \Pi ({\bf F}^{-1}({\bf 0}_{m}))$ is a solution of \eqref{without-equlities-998}, then any $({\bf x}^*, {\bf z})\in {\bf F}^{-1}({\bf 0}_{m})$ is a solution of \eqref{new-problem-554}. 
\item If  $({\bf x}^*, {\bf z}^*)\in {\bf F}^{-1}({\bf 0}_{m})$ is a solution of \eqref{new-problem-554}, then ${\bf x}^*\in \Pi ({\bf F}^{-1}({\bf 0}_{m}))$ and it is a solution of \eqref{without-equlities-998}.
\end{enumerate}
\end{thm}

\begin{proof}
On $\mathcal{M}\times\R^m$ we use the product
distance
\[
  \rho\bigl(({\bf x},{\bf z}),({\bf x}',{\bf z}')\bigr) := d({\bf x},{\bf x}') + \|{\bf z} - {\bf z}'\|,
\]
where $d$ is the Riemannian distance on $(\mathcal{M},{\bf g})$ and $\|\cdot\|$ is the Euclidean
norm on $\R^m$. This distance induces the product topology and is
equivalent to the Riemannian distance of $\widetilde {\bf g} = {\bf g} \oplus {\bf g}_{\mathrm{Euc}}$
on $\mathcal{M}\times\R^m$.\\
$(i)$ Let ${\bf x}^{*}$ be a solution of the problem \eqref{without-equlities-998}, i.e. there is $\delta>0$
such that
\[
  f_0({\bf x}^{*}) \le f_0({\bf x})\quad
  \text{for all } {\bf x}\in B({\bf x}^{*},\delta)\cap\Pi({\bf F}^{-1}({\bf 0}_{m})).
\]
Fix any ${\bf z}$ with $({\bf x}^{*},{\bf z})\in{\bf F}^{-1}({\bf 0}_{m})$ and let $({\bf x},\widetilde {\bf z})\in {\bf F}^{-1}({\bf 0}_{m})$ satisfy
$\rho\bigl(({\bf x},\widetilde {\bf z}),({\bf x}^{*},{\bf z})\bigr) < \delta$. Then
$d({\bf x},{\bf x}^{*}) \le \rho\bigl(({\bf x},\widetilde {\bf z}),({\bf x}^{*},{\bf z})\bigr) < \delta$, and
${\bf x}\in\Pi({\bf F}^{-1}({\bf 0}_{m}))$. Hence
${\bf x}\in B({\bf x}^{*},\delta)\cap\Pi({\bf F}^{-1}({\bf 0}_{m}))$, and since $F_0$ does not depend on the
slack variable,
\[
  F_0({\bf x}^{*},{\bf z}) = f_0({\bf x}^{*}) \le f_0({\bf x}) = F_0({\bf x},\widetilde {\bf z}).
\]
As $\delta$ does not depend on the chosen ${\bf z}$, every $({\bf x}^{*},{\bf z})\in{\bf F}^{-1}({\bf 0}_{m})$ is a
solution of \eqref{new-problem-554}. 

\medskip

\noindent $(ii)$ Let $({\bf x}^{*},{\bf z}^{*})$ be a solution of the problem \eqref{new-problem-554}, i.e. there is
$\delta>0$ such that
\[
  F_0({\bf x}^{*},{\bf z}^{*}) \le F_0({\bf x},{\bf z})\quad
  \text{for all } ({\bf x},{\bf z})\in {\bf F}^{-1}({\bf 0}_{m}) \text{ with }
  \rho\bigl(({\bf x},{\bf z}),({\bf x}^{*},{\bf z}^{*})\bigr) < \delta.
\]
Since $({\bf x}^{*},{\bf z}^{*})\in {\bf F}^{-1}({\bf 0}_{m})$, we have
${\bf x}^{*} = \Pi({\bf x}^{*},{\bf z}^{*}) \in \Pi({\bf F}^{-1}({\bf 0}_{m}))$.

It remains to show ${\bf x}^{*}$ is a local minimizer of \eqref{without-equlities-998}. 
We lift ${\bf x}$ to a slack vector whose signs match those of ${\bf z}^{*}$.
For $i\in \{1,\dots,m\}$ define
\begin{equation*}\label{eq:lift}
  z_i({\bf x}) := \varepsilon_i\,\sqrt{-2 f_i({\bf x})},
  \qquad
  \varepsilon_i :=
  \begin{cases}
    \text{sgn}(z_i^{*}), & z_i^{*}\neq 0,\\[2pt]
    +1, & z_i^{*}=0,
  \end{cases}
\end{equation*}
for ${\bf x}\in\Pi({\bf F}^{-1}({\bf 0}_{m}))$ in a neighborhood of ${\bf x}^{*}$. We have
$({\bf x},{\bf z}({\bf x}))\in {\bf F}^{-1}({\bf 0}_{m})$, and ${\bf z}({\bf x}^{*}) = {\bf z}^{*}$. 

Each map ${\bf x}\mapsto z_i({\bf x}) = \varepsilon_i\sqrt{-2 f_i({\bf x})}$ is continuous on the
feasible set, since $f_i$ is smooth and $-2 f_i\ge 0$ there. Hence
${\bf x}\mapsto {\bf z}({\bf x})$ is continuous at ${\bf x}^{*}$ with ${\bf z}({\bf x}^{*}) = {\bf z}^{*}$, and we may
choose $\delta' \in \left(0,\frac{\delta}{2}\right]$ such that
\[
  {\bf x}\in\Pi({\bf F}^{-1}({\bf 0}_{m})),\ d({\bf x},{\bf x}^{*}) < \delta'
  ~~\text{implies}~~ \|{\bf z}({\bf x}) - {\bf z}^{*}\| < \tfrac{\delta}{2}.
\]
For such an ${\bf x}$,
\[
  \rho\bigl(({\bf x},{\bf z}({\bf x})),({\bf x}^{*},{\bf z}^{*})\bigr)
  = d({\bf x},{\bf x}^{*}) + \|{\bf z}({\bf x}) - {\bf z}^{*}\|
  < \tfrac{\delta}{2} + \tfrac{\delta}{2} = \delta,
\]
so $({\bf x},{\bf z}({\bf x}))$ lies in the admissible $\delta$-ball of \eqref{new-problem-554}, and
therefore
\[
  F_0({\bf x}^{*},{\bf z}^{*}) = f_0({\bf x}^{*})\le f_0({\bf x}) = F_0({\bf x},{\bf z}({\bf x})).
\]
Thus $f_0({\bf x}^{*}) \le f_0({\bf x})$ for all ${\bf x}\in B(x^{*},\delta')\cap\Pi({\bf F}^{-1}({\bf 0}_{m}))$,
i.e.\ ${\bf x}^{*}$ is a solution of \eqref{without-equlities-998}.
\end{proof}

In \cite{fukushima}, it is analyzed in detail the relation between the classical Lagrange multiplier approach for the optimization problem \eqref{new-problem-554} and the classical Karush-Kuhn-Tucker approach for the optimization problem \eqref{without-equlities-998}.The same work also emphasizes the equivalence of the first- and second-order optimality conditions for these two problems.

\subsection{First order optimality conditions}

First, we address the problem associated with \eqref{new-problem-554}, which consists of determining the critical points of the function $F_0$
 subject to the constraint given by the level set  ${\bf F}^{-1}({\bf 0}_{m})$.
We solve this new problem using the embedded gradient vector field method. For this, we need that a non-empty open subset of  ${\bf F}^{-1}({\bf 0}_{m})$ to be a submanifold  of $\mathcal{M}\times \R^m$, which will be of dimension $\left(\text{dim}(\mathcal{M})+m\right)-m=N-p$. This is ensured by the Regular Value Theorem, possibly after removing the singular points, i.e. points $({\bf x}, {\bf z})$ where the gradients\footnote{To simplify the notations, when we take the gradient of a function we use the induced metric from the ambient space on the manifold where the function is defined without mentioning it explicitly.} $\nabla F_1({\bf x}, {\bf z}),\dots, \nabla F_m({\bf x}, {\bf z})$ are linearly dependent (for the example of a cone we have to remove the origin in order to have the Regular Value Theorem valid).   
We introduce the set of singular points and characterize it in the next lemma,
$${\bf F}^{-1}({\bf 0}_{m})^{\text{sing}}:=\{({\bf x,z})\in {\bf F}^{-1}({\bf 0}_{m})\,|\,\nabla F_1({\bf x}, {\bf z}),\dots, \nabla F_m({\bf x}, {\bf z})\,\,\text{ are linearly dependent}\}.$$

\begin{lemma}\label{singular-points-23}
A point $({\bf x,z})\in{\bf F}^{-1}({\bf 0}_{m})^{\text{sing}}$ if and only if there is $k\in \{1,\dots,m\}$ and $1\leq i_1<\dots<i_k\leq m$ such that $z_{i_1}=\dots=z_{i_k}=0, z_j\neq 0,\,j\notin\{i_1,\dots,i_k\}$ and  $\nabla f_{i_1}({\bf x})$, ..., $\nabla f_{i_k}({\bf x})$ are linearly dependent.
\end{lemma}

\begin{proof} The Jacobian matrix associated to the constraint functions $F_1,\dots, F_m$ is
$$ \left[D{\bf F}({\bf x},{\bf z})\right]=
\left[%
\begin{array}{ccc}
\nabla f_1({\bf x}) & ... & \nabla f_m({\bf x}) \\
z_1 & \dots & 0 \\
  \vdots & \ddots & \vdots \\
 0 & ... & z_m
\end{array}%
\right]^T.
$$
First, we suppose that $\nabla F_1({\bf x}, {\bf z}),\dots, \nabla F_m({\bf x}, {\bf z})$ are linearly dependent. Then, $\text{rank}  \left[D{\bf F}({\bf x},{\bf z})\right]<m$. This in turn implies that there is $k\in \{1,\dots,m\}$ and $1\leq i_1<\dots<i_k\leq m$ such that $z_{i_1}=\dots=z_{i_k}=0, z_j\neq 0,\,j\notin\{i_1,\dots,i_k\}$.  Without loss of generality, we can assume that $i_1=1,\dots, i_k=k$. The Jacobian matrix becomes 
$$ \left[D{\bf F}({\bf x},{\bf z})\right]=
\left[%
\begin{array}{ccc|ccc}
\nabla f_1({\bf x}) & ...& \nabla f_k({\bf x}) & \nabla f_{k+1}({\bf x}) & \dots & \nabla f_m({\bf x})  \\
{\bf 0} & \dots & {\bf 0} & {\bf 0} & \dots & {\bf 0} \\
{0} & \dots & {0} & z_{k+1} & \dots & 0 \\
\vdots & \ddots & \vdots & \vdots & \ddots & \vdots \\
 0 & \dots & 0 & 0 & \dots & z_m 
\end{array}%
\right]^T.
$$
The hypothesis implies that $\nabla f_{1}({\bf x})$, ..., $\nabla f_{k}({\bf x})$ are linearly dependent.

The converse follows immediately from the above form of Jacobian matrix.
\end{proof}

We define the set of regular points as 
$${\bf F}^{-1}({\bf 0}_{m})^{\text{reg}}:={\bf F}^{-1}({\bf 0}_{m})\backslash {\bf F}^{-1}({\bf 0}_{m})^{\text{sing}},$$
which we call the lifted feasible set.

As a consequence of the Regular Value Theorem (for a comprehensive treatment for this fundamental result, see \cite{abraham} or \cite{lee}), we have the following fundamental result. 
\begin{thm}
If~ ${\bf F}^{-1}({\bf 0}_{m})^{\text{reg}}\ne \emptyset$, then it is a manifold. 
\end{thm}

The following partitions hold 
\begin{equation}\label{partition-12}
{\bf F}^{-1}({\bf 0}_{m})^{\text{reg}}={\bf S}_0\cup \bigcup_{1\leq k\leq m}\,\,\bigcup_{1\leq i_1<\dots<i_k\leq m}{\bf S}_{(i_1,\dots,i_k)},
\end{equation}
\begin{equation}\label{partition-pi-22}
\Pi\left({\bf F}^{-1}({\bf 0}_{m})^{\text{reg}}\right)=\Pi{\bf S}_0\cup \bigcup_{1\leq k\leq m}\,\,\bigcup_{1\leq i_1<\dots<i_k\leq m}\Pi{\bf S}_{(i_1,\dots,i_k)},
\end{equation}
where 
\begin{align*}
& {\bf S}_0 =\left\{({\bf x}, {\bf z})\in{\bf F}^{-1}({\bf 0}_{m}) \,|\,z_i\neq 0,\,\,\forall i\in \{1,\dots, m\}\right\}\subset {\bf F}^{-1}({\bf 0}_{m})^{\text{reg}}, \\
 & {\bf S}_{(i_1,\dots,i_k)} =\left\{({\bf x}, {\bf z})\in{\bf F}^{-1}({\bf 0}_{m})^{\text{reg}} \,|\,z_{i_1}=\dots=z_{i_k}=0, z_j\neq 0,\,j\notin\{i_1,\dots,i_k\}\right\}, \\
 & \Pi {\bf S}_0=\left\{{\bf x}\in \mathcal{M}\,|\, f_i({\bf x})<0,\,\forall i\in\{1,\dots,m\}\right\}, \\
& \Pi {\bf S}_{(i_1,\dots,i_k)}=\left\{{\bf x}\in \Pi\left({\bf F}^{-1}({\bf 0}_{m})^{\text{reg}}\right)\,|\,f_{i_1}({\bf x})=\dots=f_{i_k}({\bf x})=0, f_j({\bf x})<0,\,j\notin\{i_1,\dots,i_k\}\right\}.
\end{align*}

Note that for each ${\bf x}\in \Pi {\bf S}_{(i_1,\dots,i_k)}$, the condition that $\nabla f_{i_1}({\bf x})$, ..., $\nabla f_{i_k}({\bf x})$ are linearly independent is satisfied. This condition is called in the literature {\it LICQ} (linear independence constraint qualification), see \cite{nocedal}, Definition 12.4, pp. 320.

\begin{lemma}\label{Lema-manifoalde} Furthermore, we have the following characterization. 
\begin{enumerate}[(i)]
\item Each element of the partition \eqref{partition-12} is a submanifold of\, $\mathcal{M}\times \R^m$.
\item Each element of the partition \eqref{partition-pi-22} is a submanifold of $\,\mathcal{M}$.
\end{enumerate}
\end{lemma}

\begin{proof}

$(i)$  The set ${\bf S}_0$ is an open subset of ${\bf F}^{-1}({\bf 0}_{m})^{\text{reg}}$, and consequently a submanifold of $\mathcal{M}\times \R^m$. Without loss of generality, we can assume that $i_1=1,\dots, i_k=k$. The set ${\bf S}_{(1,\dots,k)}=\widetilde {\bf F}^{-1}({\bf 0}_{m+k})$, where $\widetilde{\bf F}:=({\bf F},z_1,\dots,z_k):\mathcal{M}\times \R^{m}\to \R^{m+k}$.  For $({\bf x},{\bf z})\in {\bf S}_{(1,\dots, k)}$, the Jacobian matrix is
$$ \left[D\widetilde{\bf F}({\bf x},{\bf z})\right]=
\left[%
\begin{array}{ccc|ccc|c}
\nabla f_1({\bf x}) & ...& \nabla f_k({\bf x}) & \nabla f_{k+1}({\bf x}) & \dots & \nabla f_m({\bf x}) & {\bf 0} \\
{\bf 0} & \dots & {\bf 0} & {\bf 0} & \dots & {\bf 0} & \mathbb{I}_k \\
{0} & \dots & {0} & z_{k+1} & \dots & 0 & 0\\
\vdots & \ddots & \vdots & \vdots & \ddots & \vdots & \vdots \\
 0 & \dots & 0 & 0 & \dots & z_m & {0}
\end{array}%
\right]^T.
$$
By permuting the rows, we obtain maximal rank for $ \left[D\widetilde{\bf F}({\bf x},{\bf z})\right]$ if and only if the vectors $\nabla f_{1}({\bf x})$, ..., $\nabla f_{k}({\bf x})$ are linearly independent. This linear independence is guaranteed by the hypothesis that $({\bf x},{\bf z})\in {\bf F}^{-1}({\bf 0}_{m})^{\text{reg}}$.

$(ii)$  The set $ \Pi {\bf S}_0$ is an open subset of $\mathcal{M}$, and consequently a submanifold of $\mathcal{M}$. The set $(f_{i_1},\dots,f_{i_k})^{-1}({\bf 0}_k)^{\text{reg}}$ is a submanifold of $\mathcal{M}$ because the vectors  $\nabla f_{i_1}({\bf x})$, ..., $\nabla f_{i_k}({\bf x})$ are assumed linearly independent. The set $\Pi {\bf S}_{(i_1,\dots,i_k)}$ is the intersection between the submanifold  $(f_{i_1},\dots,f_{i_k})^{-1}({\bf 0}_k)^{\text{reg}}$ and the open set $\left\{{\bf x}\in \mathcal{M}\,|\, f_j({\bf x})<0,\,j\notin\{i_1,\dots,i_k\}\right\}$.
\end{proof}

{\bf In the sequel, we will work on the manifold} ${\bf F}^{-1}({\bf 0}_{m})^{\text{reg}}$, {\bf and, to simplify  the notation, we will omit the superscript} {\it reg}.

 Finding the critical points of the function $F_0$ subject to the constraints ${\bf F}^{-1}({\bf 0}_{m})$ is equivalent with finding  $({\bf x}, {\bf z})\in {\bf F}^{-1}({\bf 0}_{m})$ such that 
\begin{equation}\label{critical-points-99}
\nabla_{{\bf g}^{\text{ind}}_{{\bf F}^{-1}({\bf 0}_{m})}}\left(F_{0_{|{\bf F}^{-1}({\bf 0}_{m})}}\right)({\bf x}, {\bf z})={\bf 0},
\end{equation}
where ${\bf g}$ is the Riemannian metric on $\mathcal{M}$ and ${\bf g}^{\text{ind}}_{{\bf F}^{-1}({\bf 0}_{m})}$ is the induced metric on ${\bf F}^{-1}({\bf 0}_{m})$.

To solve this constrained optimization problem, we employ the embedded gradient vector field method developed in \cite{birtea-comanescu} and \cite{Birtea-Comanescu-Hessian}. 
It has been shown that for $({\bf x}, {\bf z})\in {\bf F}^{-1}({\bf 0}_{m})$, the following equality holds
\begin{equation}\label{lagrangean-990}
 \nabla_{{\bf g}^{\text{ind}}_{{\bf F}^{-1}({\bf 0}_{m})}}\left(F_{0_{|{\bf F}^{-1}({\bf 0}_{m})}}\right)({\bf x}, {\bf z})=\partial F_0({\bf x}, {\bf z})  := \nabla  F_0({\bf x}, {\bf z})-\sum_{\alpha=1}^m \sigma_{\alpha}({\bf x},{\bf z})  \nabla F_{\alpha}({\bf x}, {\bf z}), 
\end{equation}
where the Lagrange multiplier functions\footnote{These functions are defined on the open subset of regular points of the function ${\bf F}$.} $\sigma_{\alpha}:\mathcal{M}\times \R^m \to \R$ associated with the constraint functions $F_1, \dots, F_m$, are defined by the (Cramer's rule, see \cite{birtea-comanescu} ) formula
\begin{equation*}\label{sigma-101}
\sigma_{\alpha}({\bf x}, {\bf z}):=\frac{\det\left(\sm_{(F_1,\ldots ,F_{\alpha-1},F_{\alpha}, F_{\alpha+1},\dots,F_m)}^{(F_1,\ldots , F_{\alpha-1},F_0, F_{\alpha+1},...,F_m)}({\bf x}, {\bf z})\right)}{\det\left(\sm_{(F_1,\ldots ,F_{\alpha-1},F_{\alpha}, F_{\alpha+1},\dots,F_m)}^{(F_1,\ldots , F_{\alpha-1},F_{\alpha}, F_{\alpha+1},...,F_m)}({\bf x}, {\bf z})\right)},
\end{equation*}
where\footnote{$\widetilde{\bf g}:={\bf g}+{\bf g}_{\text{Euc}}$ is the Riemannian metric on $\mathcal{M}\times \R^m$.}
\begin{equation*}\label{sigma}
\sm_{(f_1,...,f_r)}^{(h_1,...,h_s)}:=\left[%
\begin{array}{cccc}
  \widetilde{\bf g}(\nabla h_1,\nabla f_{1}) & ... & \widetilde{\bf g}(\nabla h_s,\nabla f_{1}) \\
  \vdots & \ddots & \vdots \\
  \widetilde{\bf g}(\nabla h_1,\nabla f_{r}) & ... & \widetilde{\bf g}(\nabla h_s,\nabla f_{r})
\end{array}%
\right].
\end{equation*}

Applying this method to our case, a straightforward computation gives us the embedded gradient vector field
\begin{align*}
\partial F_0({\bf x}, {\bf z})
 =\left( \nabla f_0({\bf x})-\sum_{\alpha=1}^m \sigma_{\alpha}({\bf x},{\bf z})  \nabla f_{\alpha}({\bf x}), - \sigma_{1}({\bf x},{\bf z}) z_1,\dots,  - \sigma_{m}({\bf x},{\bf z}) z_m \right).
\end{align*}
This gradient vector field is defined on the open set of $\mathcal{M}\times \R^m$ given by the regular points of ${\bf F}$, moreover it is tangent to every regular leaf generated by ${\bf F}$, in particular is tangent to ${\bf F}^{-1}({\bf 0}_{m})$.

\begin{rem}
Our choice of sign in \eqref{lagrangean-990} (coming from geometric reasons) are opposite with the standard choice, in the Karush-Kuhn-Tucker optimization problem,  for the ''augmented Lagrangian function'', thus producing a sign inversion for these multipliers. 
\end{rem}

As a consequence of the embedded gradient vector field method, we have a first characterization for the solutions of problem \eqref{critical-points-99}. 

\begin{thm}
\begin{enumerate}[(i)]
\item A point $({\bf x^*}, {\bf z^*})\in {\bf F}^{-1}({\bf 0}_{m})$ is a solution of \eqref{critical-points-99} {\bf if and only if }
\begin{equation*}\label{ecuatii-sigma-121}
\begin{cases}
  \nabla f_0({\bf x}^*)-\sum\limits_{\alpha=1}^m \sigma_{\alpha}({\bf x}^*,{\bf z}^*)  \nabla f_{\alpha}({\bf x}^*)={\bf 0}_{N-p} \\
 \sigma_{\alpha}({\bf x}^*,{\bf z}^*) z^*_{\alpha}=0,\,\,\,\alpha\in \{1,\dots,m\}.
\end{cases}
\end{equation*}
\item A point $({\bf x^*}, {\bf z^*})\in {\bf S}_{(i_1,\dots,i_k)}\subset {\bf F}^{-1}({\bf 0}_{m})$ is a solution of \eqref{critical-points-99} {\bf if and only if }
\begin{equation}\label{ecuatii-sigma-133}
\begin{cases}
  \nabla f_0({\bf x}^*)-\sum\limits_{i\in \{i_1,\dots,i_k\}} \sigma_{i}({\bf x}^*,{\bf z}^*)  \nabla f_{i}({\bf x}^*)={\bf 0}_{N-p} \\
 \sigma_{j}({\bf x}^*,{\bf z}^*)=0,\,\,\,j\notin \{i_1,\dots, i_k\}.
\end{cases}
\end{equation}
\item A point $({\bf x^*}, {\bf z^*})\in {\bf S}_0\subset {\bf F}^{-1}({\bf 0}_{m})$ is a solution of \eqref{critical-points-99} {\bf if and only if }
$
  \nabla f_0({\bf x}^*)={\bf 0}_{N-p}.
$
\end{enumerate}
\end{thm}

In  the literature, the second set of equations in \eqref{ecuatii-sigma-121} is called {\it complementary conditions} of the classical Karush-Kuhn-Tucker approach, see \cite{nocedal}, pp.321.

Our next goal is to eliminate the slack variables $z_1,\dots,z_m$, which corresponds to projecting the embedded gradient vector field $\partial F_0$ along the projection map $\Pi({\bf x}, {\bf z})={\bf x}$. {\bf The issue is that the vector field $\partial F_0$ is not projectable along the map $\Pi$}. 

\begin{thm}\label{critice-mu-11}
A point $({\bf x^*}, {\bf z^*})\in {\bf S}_{(i_1,\dots,i_k)}$ is a solution of \eqref{critical-points-99} (equivalently, a solution of \eqref{ecuatii-sigma-133}) {\bf if and only if } 
${\bf x}^*\in \Pi {\bf S}_{(i_1,\dots,i_k)}=\left\{{\bf x}\in \mathcal{M}\,|\,f_{i_1}({\bf x})=\dots=f_{i_k}({\bf x})=0, f_j({\bf x})<0,\,j\notin\{i_1,\dots,i_k\}\right\}$ is a solution of
\begin{equation*}\label{ecuatie-mu-9997-v2}
\nabla f_0({\bf x}^*)-\sum_{i\in \{i_1,\dots, i_k\}} \mu_{_{i}}({\bf x}^*)  \nabla f_{_{i}}({\bf x}^*)={\bf 0}_{N-p},
\end{equation*}
where the Lagrange multiplier functions\footnote{These functions are defined on the open subset of regular points of the function $(f_{i_1}, \dots, f_{i_k})$.} $\mu_{_{i_{\alpha}}}: \Pi {\bf S}_{(i_1,\dots,i_k)}\to \R$ associated with the active constraints $f_{i_1}, \dots, f_{i_k}$, are given by
\begin{equation}\label{formule-mu-v2}
\mu_{_{i_{\alpha}}}({\bf x}):=\frac{\det\left(\sm_{(f_{i_{_1}},\ldots ,f_{i_{_{\alpha-1}}},f_{i_{_{\alpha}}}, f_{i_{_{\alpha+1}}},\dots,f_{i_{_k}})}^{(f_{i_{_1}},\ldots ,f_{i_{_{\alpha-1}}},f_{_0}, f_{i_{_{\alpha+1}}},\dots,f_{i_{_k}})}\right)}{\det\left(\sm_{(f_{i_{_1}},\ldots ,f_{i_{_{\alpha-1}}},f_{i_{_{\alpha}}}, f_{i_{_{\alpha+1}}},\dots,f_{i_{_k}})}^{(f_{i_{_1}},\ldots ,f_{i_{_{\alpha-1}}},f_{i_{_{\alpha}}}, f_{i_{_{\alpha+1}}},\dots,f_{i_{_k}})}\right)}({\bf x}),\,\,\alpha\in \{1,\dots,k\}.
\end{equation}
Moreover, we have the equalities  $\sigma_i({\bf x^*}, {\bf z^*})=\mu_{i}({\bf x}^*)$, $i\in \{i_1, \dots, i_k\}$, and $ \sigma_{j}({\bf x}^*,{\bf z}^*)=0,\,j\notin \{i_1,\dots, i_k\}$.
\end{thm}

\begin{proof}
First we prove the implication $''\Rightarrow''$. Let $({\bf x^*}, {\bf z^*})\in {\bf S}_{(i_1,\dots,i_k)}$ be a solution of \eqref{ecuatii-sigma-133} which implies the following vector equality
$$\sum\limits_{i\in \{i_1,\dots,i_k\}} \sigma_{i}({\bf x}^*,{\bf z}^*)  \nabla f_{i}({\bf x}^*)= \nabla f_0({\bf x}^*).$$

By scalar multiplying each member of the above vector equality with $\nabla f_{i_1}({\bf x}^*), \dots, \nabla f_{i_k}({\bf x}^*)$, we obtain that the vector $\left( \sigma_{i_1}({\bf x}^*,{\bf z}^*), \dots,  \sigma_{i_k}({\bf x}^*,{\bf z}^*)\right)$ is a solution of the linear system
$$\sm_{f_{i_1},\dots, f_{i_k}}^{f_{i_1},\dots, f_{i_k}}({\bf x}^*) \left[%
\begin{array}{c}
\sigma_{i_1}({\bf x}^*,{\bf z}^*) \\
\vdots \\
\sigma_{i_k}({\bf x}^*,{\bf z}^*)
\end{array}%
\right]= 
\left[%
\begin{array}{c}
{\bf g}(\nabla f_0({\bf x}^*), \nabla f_{i_1}({\bf x}^*)) \\
\vdots \\
{\bf g}(\nabla f_0({\bf x}^*), \nabla f_{i_k}({\bf x}^*))
\end{array}%
\right].
 $$
 By our hypotheses, the matrix $\sm_{f_{i_1},\dots, f_{i_k}}^{f_{i_1},\dots, f_{i_k}}({\bf x}^*)$ is non-degenerate and consequently, by Cramer's rule, the vector $\left( \mu_{i_1}({\bf x}^*), \dots,  \mu_{i_k}({\bf x}^*)\right)$ is the unique solution, where $\mu_{_{i_{\alpha}}}({\bf x}^*)$ are given by \eqref{formule-mu-v2}. 
 
 For the converse implication $''\Leftarrow''$, we assume without loss of generality that $i_1=1,\dots i_k=k$. For $({\bf x},{\bf z})\in {\bf F}^{-1}({\bf 0}_{m})$, we have the computation:
\begin{align*}
\sigma_{\alpha}({\bf x}, {\bf z}) & = \frac{\det\left(\sm_{(f_1,\ldots ,f_{\alpha-1},f_{\alpha}, f_{\alpha+1},\dots,f_m)}^{(f_1,\ldots , f_{\alpha-1},f_0, f_{\alpha+1},...,f_m)}({\bf x})+\text{diag}(z_1^2, \dots, z_{\alpha-1}^2,0,z_{\alpha+1}^2,\dots,z_{m}^2)\right)}{\det\left(\sm_{(f_1,\ldots ,f_{\alpha-1},f_{\alpha}, f_{\alpha+1},\dots,f_m)}^{(f_1,\ldots , f_{\alpha-1},f_{\alpha}, f_{\alpha+1},...,f_m)}({\bf x})+\text{diag}(z_1^2, \dots, z_{\alpha-1}^2,z_{\alpha}^2,z_{\alpha+1}^2,\dots,z_{m}^2)\right)}, \,\,\alpha\in \{1,\dots m\}.
\end{align*}
For ${\bf x}^*\in \Pi {\bf S}_{(i_1,\dots,i_k)}$ and ${\bf z}^*\in\R^m$ so that $({\bf x^*}, {\bf z^*})\in {\bf S}_{(i_1,\dots,i_k)}\subset {\bf F}^{-1}({\bf 0}_{m})$, we have 
 \begin{align*}
\sigma_{j}({\bf x}^*, {\bf z}^*) & = \frac{\det\left(\sm_{(f_1,\ldots ,f_{k},f_{k+1},\dots, f_{j},\dots, f_m)}^{(f_1,\ldots , f_{k},f_{k+1}, \dots, f_{0},...,f_m)}({\bf x}^*)+\text{diag}(0, \dots, 0,z_{k+1}^2,\dots, 0, \dots, z_{m}^2)\right)}{\det\left(\sm_{(f_1,\ldots ,f_{k},f_{k+1},\dots, f_{j},\dots, f_m)}^{(f_1,\ldots , f_{k},f_{k+1}, \dots, f_{j},...,f_m)}({\bf x}^*)+\text{diag}(0, \dots, 0,z_{k+1}^2,\dots, z_j^2, \dots, z_{m}^2)\right)},
\end{align*}
for all $j\notin\{1,\dots,k\}$. In the matrix of the numerator the $j$-th column is a linear combination of the first $k$ columns. This fact is a consequence of equation \eqref{ecuatie-mu-9997-v2} and scalar multiplication argument as above. 

By the definition of the Lagrange multiplier functions, the vector $\left( \sigma_{1}({\bf x},{\bf z}), \dots,  \sigma_{m}({\bf x},{\bf z})\right)$ is a solution of the linear system
$$\sm_{F_{1},\dots, F_{m}}^{F_{1},\dots, F_{m}}({\bf x},{\bf z}) \left[%
\begin{array}{c}
\sigma_{1}({\bf x},{\bf z}) \\
\vdots \\
\sigma_{m}({\bf x},{\bf z})
\end{array}%
\right]= 
\left[%
\begin{array}{c}
{\bf g}(\nabla F_0({\bf x},{\bf z}), \nabla F_{1}({\bf x},{\bf z})) \\
\vdots \\
{\bf g}(\nabla F_0({\bf x},{\bf z}), \nabla F_{m}({\bf x},{\bf z}))
\end{array}%
\right].
 $$
 For $({\bf x^*}, {\bf z^*})\in {\bf S}_{(i_1,\dots,i_k)}\subset {\bf F}^{-1}({\bf 0}_{m})$, the above system becomes
 $$
 \left[%
 \begin{array}{cc}
 \sm_{f_{1},\dots, f_{k}}^{f_{1},\dots, f_{k}}({\bf x}^*) & A({\bf x^*}, {\bf z^*}) \\
 A^T({\bf x^*}, {\bf z^*}) & B({\bf x^*}, {\bf z^*})
 \end{array}
 \right]
 \left[%
\begin{array}{c}
\sigma_{1}({\bf x}^*,{\bf z}^*) \\
\vdots \\
\sigma_{k}({\bf x}^*,{\bf z}^*) \\
0 \\
\vdots \\
0
\end{array}%
\right]= 
\left[%
\begin{array}{c}
{\bf g}(\nabla f_0({\bf x}^*), \nabla f_{1}({\bf x}^*)) \\
\vdots \\
{\bf g}(\nabla f_0({\bf x}^*), \nabla f_{k}({\bf x}^*)) \\
{\bf g}(\nabla f_0({\bf x}^*), \nabla f_{k+1}({\bf x}^*)) \\
\vdots \\
{\bf g}(\nabla f_0({\bf x}^*), \nabla f_{m}({\bf x}^*)) \\
\end{array}%
\right].
 $$
 Consequently, the vector $\left( \sigma_{1}({\bf x}^*,{\bf z}^*), \dots,  \sigma_{k}({\bf x}^*,{\bf z}^*)\right)$ is a solution of the linear system
$$\sm_{f_{1},\dots, f_{k}}^{f_{1},\dots, f_{k}}({\bf x}^*) \left[%
\begin{array}{c}
\sigma_{1}({\bf x}^*,{\bf z}^*) \\
\vdots \\
\sigma_{k}({\bf x}^*,{\bf z}^*)
\end{array}%
\right]= 
\left[%
\begin{array}{c}
{\bf g}(\nabla f_0({\bf x}^*), \nabla f_{1}({\bf x}^*)) \\
\vdots \\
{\bf g}(\nabla f_0({\bf x}^*), \nabla f_{k}({\bf x}^*))
\end{array}%
\right],
 $$
which is the same linear system verified by the vector $\left( \mu_{1}({\bf x}^*), \dots,  \mu_{k}({\bf x}^*)\right)$. By the uniqueness of the solution, we obtain the announced result. 
 \end{proof}

\subsection{Second order optimality conditions} 

    First we discuss the second order conditions for the optimization problem \eqref{new-problem-554}. We need to compute $\text{Hess}\left(F_{0_{|{\bf F}^{-1}({\bf 0}_{m})}}\right)$. According to the result proved in \cite{Birtea-Comanescu-Hessian}, for every $({\bf x}, {\bf z})\in {\bf F}^{-1}({\bf 0}_{m})$,  we have the following formula
$$\text{Hess}\left(F_{0_{|{\bf F}^{-1}({\bf 0}_{m})}}\right)({\bf x}, {\bf z})  := \left(\text{Hess}  F_0({\bf x}, {\bf z})-\sum_{\alpha=1}^m \sigma_{\alpha}({\bf x},{\bf z})  \text{Hess} F_{\alpha}({\bf x}, {\bf z}) \right)_{{|T_{({\bf x,z})}{\bf F}^{-1}({\bf 0}_{m})\times T_{({\bf x,z})}{\bf F}^{-1}({\bf 0}_{m})}}.$$

The tangent space is defined by
\begin{align*}
T_{({\bf x,z})}{\bf F}^{-1}({\bf 0}_{m}) & =\left\{({\bf v}, {\bf w})\in T_{\bf x}\mathcal{M}\times \R^m\,|\,({\bf v}, {\bf w})\perp \nabla F_{\alpha}({\bf x,z}),\,\,\forall \alpha\in \{1,\dots,m\}\right\}\\
& = \left\{({\bf v}, {\bf w})\in T_{\bf x}\mathcal{M}\times \R^m\,|\,{\bf g}({\bf v},\nabla f_{\alpha}({\bf x}))+w_{\alpha}z_{\alpha}=0,\,\,\forall \alpha\in \{1,\dots,m\}\right\},
\end{align*}
where we use the canonical basis for $\R^m$ and write ${\bf w}=(w_1, \dots, w_m)^T$. 

Fixing a basis for $T_{\bf x}\mathcal{M}$, we have the following matrix form 
\begin{align*}
& \left[\text{Hess}\left(F_{0_{|{\bf F}^{-1}({\bf 0}_{m})}}\right)({\bf x}, {\bf z})\right]= \\
& \left[%
\begin{array}{c|ccc}
\text{Hess} f_0({\bf x})-\sum\limits_{\alpha=1}^m \sigma_{\alpha}({\bf x},{\bf z})  \text{Hess} f_{\alpha}({\bf x}) &  & {\bf 0} &   \\
\hline
 & -\sigma_1({\bf x},{\bf z}) & \dots & 0 \\
 {\bf 0} & \vdots & \ddots & \vdots \\
&  0 & \dots & -\sigma_{m}({\bf x},{\bf z})
\end{array}%
\right]_{{|T_{({\bf x,z})}{\bf F}^{-1}({\bf 0}_{m})\times T_{({\bf x,z})}{\bf F}^{-1}({\bf 0}_{m})}}.
\end{align*}
For a tangent vector $({\bf v}, {\bf w})\in T_{({\bf x,z})}{\bf F}^{-1}({\bf 0}_{m})$, we have 
\begin{align}
 [{\bf v}^T, {\bf w}^T] & \left[\text{Hess}\left(F_{0_{|{\bf F}^{-1}({\bf 0}_{m})}}\right)({\bf x}, {\bf z})\right]  \left[%
\begin{array}{c}
{\bf v} \\
{\bf w} \end{array}%
\right] =\nonumber \\
& [{\bf v}^T] \left[\text{Hess} f_0({\bf x})-\sum\limits_{\alpha=1}^m \sigma_{\alpha}({\bf x},{\bf z})  \text{Hess} f_{\alpha}({\bf x})\right] \left[%
\begin{array}{c}
{\bf v}  \end{array}%
\right]-\sum\limits_{\alpha=1}^m \sigma_{\alpha}({\bf x},{\bf z})w^2_{\alpha}\label{relatie-hessiene}
\end{align}

\begin{thm}
Let $({\bf x^*}, {\bf z^*})\in {\bf S}_{(i_1,\dots,i_k)}$ be a solution of \eqref{critical-points-99}. Then 
\begin{enumerate}[(i)]
\item The matrix $ \left[\emph{Hess}\left(F_{0_{|{\bf F}^{-1}({\bf 0}_{m})}}\right)({\bf x}^*, {\bf z}^*)\right]$ is positive semi-definite {\bf if and only if} the matrix $$\left[\emph{Hess} f_0({\bf x}^*)-\sum\limits_{i\in \{i_1, \dots, i_k\}} \mu_{i}({\bf x}^*)  \emph{Hess} f_{i}({\bf x}^*)\right]_{{|T_{{\bf x}^*}\Pi {\bf S}_{(i_1,\dots,i_k)}\times T_{{\bf x}^*}\Pi {\bf S}_{(i_1,\dots,i_k)}}} $$ 
is positive semi-definite and $\mu_{i}({\bf x}^*)\leq 0$, for all $i\in \{i_1, \dots, i_k\}$.
\item The matrix $ \left[\emph{Hess}\left(F_{0_{|{\bf F}^{-1}({\bf 0}_{m})}}\right)({\bf x}^*, {\bf z}^*)\right]$ is positive definite {\bf if and only if} the matrix $$\left[\emph{Hess} f_0({\bf x}^*)-\sum\limits_{i\in \{i_1, \dots, i_k\}} \mu_{i}({\bf x}^*)  \emph{Hess} f_{i}({\bf x}^*)\right]_{{|T_{{\bf x}^*}\Pi {\bf S}_{(i_1,\dots,i_k)}\times T_{{\bf x}^*}\Pi {\bf S}_{(i_1,\dots,i_k)}}}$$ is positive definite and $\mu_{i}({\bf x}^*)< 0$, for all $i\in \{i_1, \dots, i_k\}$.
\end{enumerate}
\end{thm}

\begin{proof} From Theorem \ref{critice-mu-11}, for a critical point $({\bf x^*}, {\bf z^*})\in {\bf S}_{(i_1,\dots,i_k)}$ (i.e. a solution of \eqref{critical-points-99}), we have the equalities $\sigma_i({\bf x^*}, {\bf z^*})=\mu_{i}({\bf x}^*)$, $i\in \{i_1, \dots, i_k\}$, and $ \sigma_{j}({\bf x}^*,{\bf z}^*)=0,\,j\notin \{i_1,\dots, i_k\}$.

A vector $({\bf v}, {\bf w})\in T_{({\bf x}^*,{\bf z}^*)}{\bf F}^{-1}({\bf 0}_{m})$ is characterized by the equations 
\begin{align*}
 & {\bf g}({\bf v},\nabla f_{i}({\bf x}^*))=0,\,\,\forall i\in \{i_1,\dots, i_k\};
 \\
& {\bf g}({\bf v},\nabla f_{j}({\bf x}^*))+w_{j}z^*_{j}=0,\,\,\forall j\notin \{i_1,\dots, i_k\}.
\end{align*}
It is important to discuss the geometric interpretation of the above two equations. The first equation gives us a restriction for the tangent vector ${\bf v}$, more precisely ${\bf v}\in T_{{\bf x}^*}\Pi {\bf S}_{(i_1,\dots,i_k)}\subset T_{{\bf x}^*}\mathcal{M}$ and no restriction in the components $w_{i_1}, \dots, w_{i_k}$ of the vector ${\bf w}$. The second equation gives us a restriction for the components $w_j$ of ${\bf w}$ and no further restriction for the vector ${\bf v}$. 

The relation \eqref{relatie-hessiene} becomes
\begin{align*}
 [{\bf v}^T, {\bf w}^T] & \left[\text{Hess}\left(F_{0_{|{\bf F}^{-1}({\bf 0}_{m})}}\right)({\bf x}^*, {\bf z}^*)\right]  \left[%
\begin{array}{c}
{\bf v} \\
{\bf w} \end{array}%
\right] =\nonumber \\
& [{\bf v}^T] \left[\text{Hess} f_0({\bf x}^*)-\sum\limits_{i\in \{i_1, \dots, i_k\}} \mu_{i}({\bf x}^*)  \text{Hess} f_{i}({\bf x}^*)\right] \left[%
\begin{array}{c}
{\bf v}  \end{array}%
\right]-\sum\limits_{i\in \{i_1, \dots, i_k\}} \mu_{i}({\bf x}^*)w^2_{i}.
\end{align*}
Applying the above formula for the cases $({\bf v}, {\bf 0})$ with ${\bf v}\in T_{{\bf x}^*}\Pi {\bf S}_{(i_1,\dots,i_k)}\subset T_{{\bf x}^*}\mathcal{M}$ and $({\bf 0}, {\bf w})$ with $w_{i_1}, \dots, w_{i_k}$ arbitrary, we obtain the announced results. 
\end{proof}

\begin{figure}
\centering
\begin{tikzpicture}[
    box/.style={
        draw,
        line width=1pt,
        minimum height=1.4cm,
        align=center,
        text width=6cm,
        inner sep=6pt
    }
]

\node[box, minimum width=12cm] (A) at (0,0) {
$\left.\begin{array}{cc}
\hbox{argmin}& f_0({\bf x})\\
{\bf x}\in \mathcal{M}\\
\hbox{s.t.}&f_1({\bf x})\leq 0\\
&\vdots\\
&f_m({\bf x})\leq 0
\end{array}\right.$
};

\node at (0,-2) {$\scalebox{2}{$\Updownarrow$}$};

\node[box, minimum width=12cm] (B) at (0,-3.1) {
$\left.\begin{array}{cc}
\hbox{argmin}& F_0({\bf x,z})\\
({\bf x,z})\in {\bf F}^{-1}({\bf 0}_m)&
\end{array}\right.$
};

\node at (-3.8,-4.2) {$\scalebox{2}{$\Uparrow$}$};

\node at (3.8,-4.2) {$\scalebox{2}{$\Downarrow$}$};

\node[box, minimum width=5.3cm, minimum height=2.4cm] at (-3.8,-5.9) {
    $({\bf x}^*,{\bf z}^*)\in {\bf F}^{-1}({\bf 0}_m)$\\
    $\hbox{s.t.}$\\
    $\nabla\left(F_{0_{|{\bf F}^{-1}({\bf 0}_{m})}}\right)({\bf x}^*, {\bf z}^*)={\bf 0}$\\
    $\hbox{Hess}\left(F_{0_{|{\bf F}^{-1}({\bf 0}_{m})}}\right)({\bf x}^*, {\bf z}^*)>0 $
};

\node[box, minimum width=5.3cm, minimum height=2.4cm] at (3.8,-5.9) {
    $({\bf x}^*,{\bf z}^*)\in {\bf F}^{-1}({\bf 0}_m)$\\
    $\hbox{s.t.}$\\
    $\nabla\left(F_{0_{|{\bf F}^{-1}({\bf 0}_{m})}}\right)({\bf x}^*, {\bf z}^*)={\bf 0}$\\
    $\hbox{Hess}\left(F_{0_{|{\bf F}^{-1}({\bf 0}_{m})}}\right)({\bf x}^*, {\bf z}^*)\geq 0 $
};

\node at (-3.8,-7.7) {$\scalebox{2}{$\Updownarrow$}$};

\node at (3.8,-7.7) {$\scalebox{2}{$\Updownarrow$}$};

\node[box, minimum width=5.2cm, minimum height=2.4cm] at (-3.8,-10.05) {
    $({\bf x}^*,{\bf z}^*)\in {\bf F}^{-1}({\bf 0}_m)$\\
    $\hbox{s.t.}$\\
    $\left(\nabla  F_0-\sum\limits_{\alpha=1}^m \sigma_{\alpha}  \nabla F_{\alpha}\right)\left({\bf x}^*,{\bf z}^*\right)=0$\\
   {\footnotesize  $\left(\hbox{Hess}  F_0-\sum\limits_{\alpha=1}^m \sigma_{\alpha}  \hbox{Hess} F_{\alpha} \right)_{|\mathcal{T}\times \mathcal{T}}
    ({\bf x}^*, {\bf z}^*)>0 
    $} \\
   ~\\
    $\mathcal{T}:=T_{({\bf x^*,z^*})}{\bf F}^{-1}({\bf 0}_{m})$
};

\node[box, minimum width=5.2cm, minimum height=2.4cm] at (3.8,-10.05) {
    $({\bf x}^*,{\bf z}^*)\in {\bf F}^{-1}({\bf 0}_m)$\\
    $\hbox{s.t.}$\\
    $\left(\nabla  F_0-\sum\limits_{\alpha=1}^m \sigma_{\alpha}  \nabla F_{\alpha}\right)\left({\bf x}^*,{\bf z}^*\right)=0$\\
  {\footnotesize  $\left(\hbox{Hess}  F_0-\sum\limits_{\alpha=1}^m \sigma_{\alpha}  \hbox{Hess} F_{\alpha} \right)_{|\mathcal{T}\times \mathcal{T}}({\bf x}^*, {\bf z}^*)\geq 0
    $} \\
      ~\\
    $\mathcal{T}:=T_{({\bf x^*,z^*})}{\bf F}^{-1}({\bf 0}_{m})$
};

\node at (-3.8,-12.4) {$\scalebox{2}{$\Updownarrow$}$};

\node at (3.8,-12.4) {$\scalebox{2}{$\Updownarrow$}$};

\node[box, minimum width=6cm, minimum height=2.4cm] at (-3.8,-15.3) {
    ${\bf x^*}\in \Pi{\bf S}_{(i_1,\dots,i_k)}$\\
    $\hbox{s.t.}$\\
    $\left(\nabla  f_0-\sum\limits_{i\in I
    } \mu_i\nabla f_i\right)({\bf x}^*)=0$\\
    {$\footnotesize \left(\hbox{Hess}  f_0-\sum\limits_  {i\in I
    } \mu_i\hbox{Hess} f_{i} \right)_{|\Theta\times \Theta}({\bf x}^*)> 0
    $}\\
    ~\\
    $\Theta:=T_{{\bf x^*}}\Pi{\bf S}_{(i_1,\dots,i_k)}$ \\
    ~\\
    $\mu_i({\bf x}^*)<0,~~(\forall) i\in I:=\{i_1,\dots,i_k\} $
};

\node[box, minimum width=6cm, minimum height=2.4cm] at (3.8,-15.3) {
    ${\bf x^*}\in \Pi{\bf S}_{(i_1,\dots,i_k)}$\\
    $\hbox{s.t.}$\\
    $\left(\nabla  f_0-\sum\limits_{i\in I
    } \mu_i\nabla f_i\right)({\bf x}^*)=0$\\
    {\footnotesize $\left(\hbox{Hess}  f_0-\sum\limits_  {i\in I
    } \mu_i\hbox{Hess} f_{i} \right)_{|\Theta\times \Theta}({\bf x}^*)\geq  0
    $}\\
    ~\\
     $\Theta:=T_{{\bf x^*}}\Pi{\bf S}_{(i_1,\dots,i_k)}$ \\
     ~\\
    $\mu_i({\bf x}^*)\leq 0,~~(\forall) i\in I:=\{i_1,\dots,i_k\} $
};
\end{tikzpicture}
\caption{Overview of the necessary, respectively sufficient conditions for the argmin problem based on first- and second-order tests.}
\end{figure}
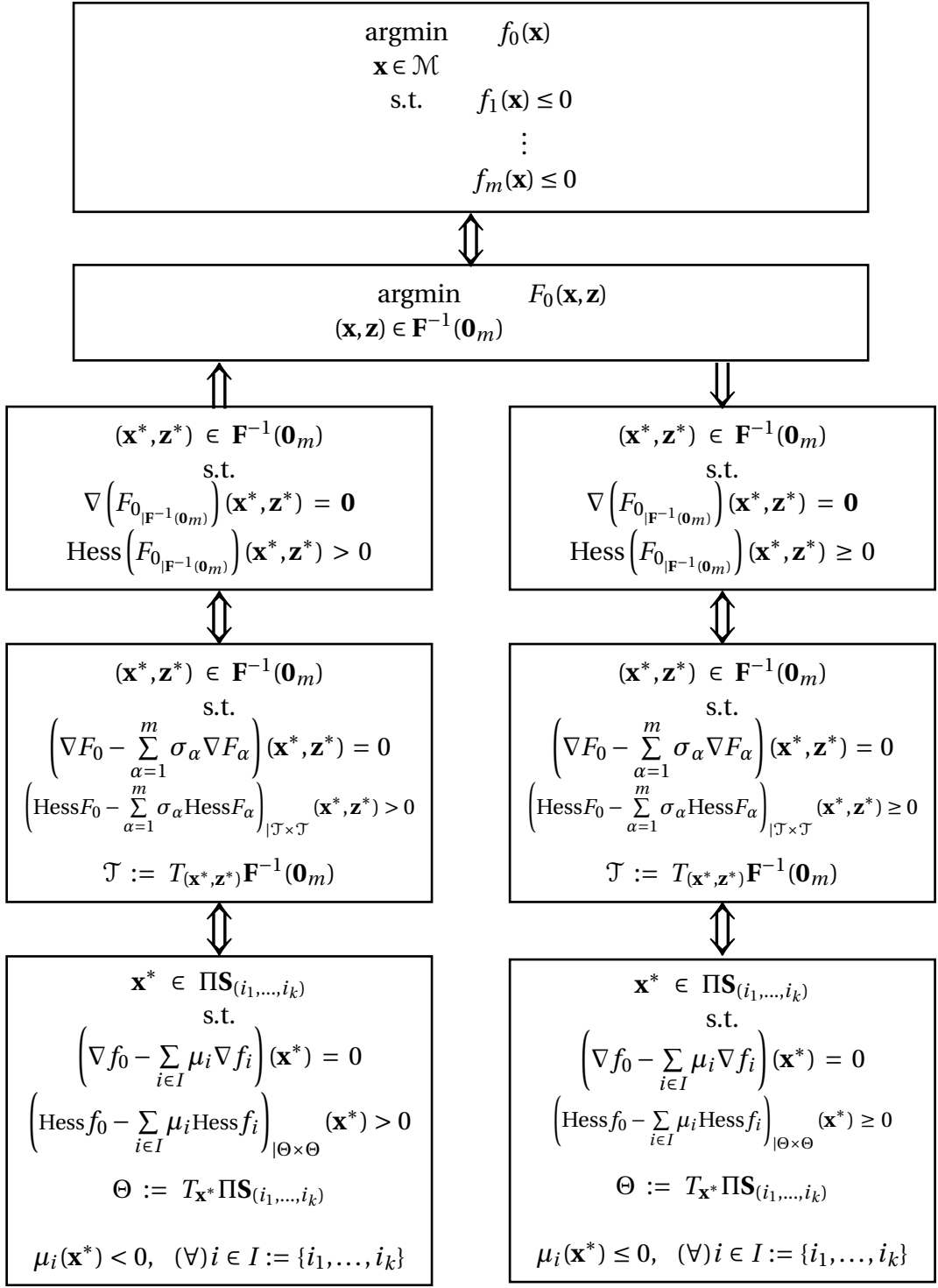

\newpage

\subsection{Algorithmic procedure for searching local minima}

Based on the previous first- and second-order results, in what follows we describe the succession of steps in solving the argmin problem \eqref{without-equlities-998}.\medskip 

{\bf I.} Determine the manifold ${\bf F}^{-1}({\bf 0}_{m})^{\text{reg}}$ and the elements of the partitions \eqref{partition-12} and \eqref{partition-pi-22}.

{\bf II.} Solve the problem on the open stratum $\Pi {\bf S}_0=\{{\bf x}\in \mathcal{M}\,|\,f_i({\bf x})< 0,\,i\in \{1,\dots,m\}\}$. This is an unconstrained optimization problem. \medskip

{\bf III.} For each element $\Pi {\bf S}_{(i_1,\dots,i_k)}$ of the partition \eqref{partition-pi-22} proceed as follows:\medskip

{\bf Step 1.} On the open set 
$$\mathcal{O}_{(i_1,\dots, i_k)}:=\{{\bf x}\in \mathcal{M}\,|\, \nabla f_{i_1}({\bf x}), \dots, \nabla f_{i_k}({\bf x})\, \text{are linearly independent}\}\subset \mathcal{M},$$ construct the Lagrange multiplier functions $\mu_{i_1}, \dots, \mu_{i_k}$ defined in \eqref{formule-mu-v2}. Using the active constraints $f_{i_1}, \dots, f_{i_k}$, in some cases these functions can be further simplified, see \cite{first-order}, \cite{second-order}, and \cite{simplectic}.

Find all ${\bf x}\in \mathcal{O}{(i_1,\dots, i_k)}$ that solve the system of equations 
\begin{equation}\label{sistem-de-rezolvat}
\nabla f_0({\bf x})-\sum\limits_{i\in \{i_1,\dots,i_k\}}\mu_{i}({\bf x})\nabla f_i({\bf x})={\bf 0}.
\end{equation}
Since $\Pi {\bf S}_{(i_1,\dots, i_k)}\subset \mathcal{O}{(i_1,\dots, i_k)}$, select the solutions ${\bf x}^*$ that verify the active constraints $f_{i_1}({\bf x})=0, \dots,$ $f_{i_k}({\bf x})=0$ and the strict inequalities $f_j({\bf x}^*)<0,\,j\notin\{i_1,\dots,i_k\}$. This is the set of candidate solutions (coming from the stratum where $f_{i_1},\dots, f_{i_k}$ are the active constraints) for the initial problem \eqref{without-equlities-998}.\medskip

{\bf Step 2.} For each ${\bf x}^*$ in the set of candidate solutions found in the previous step, compute the set of (scalar) Lagrange multipliers $\{\mu_{i_1}({\bf x}^*), \dots, \mu_{i_k}({\bf x}^*)\}$. 

2.1. If $\mu_{i}({\bf x}^*)<0$, for all $i\in \{i_1, \dots, i_k\}$, then check the positive definiteness, i.e. 
\begin{equation}\label{conditie-Hessiana-22}
\left[\text{Hess} f_0({\bf x}^*)-\sum\limits_{i\in \{i_1, \dots, i_k\}} \mu_{i}({\bf x}^*)  \text{Hess} f_{i}({\bf x}^*)\right]_{{|T_{{\bf x}^*}\Pi {\bf S}_{(i_1,\dots,i_k)}\times T_{{\bf x}^*}\Pi {\bf S}_{(i_1,\dots,i_k)}}}>0.
\end{equation}
If \eqref{conditie-Hessiana-22} is also verified, then such a point ${\bf x}^*$ is a solution for the initial problem  \eqref{without-equlities-998}.

2.2. If there exists $i\in \{i_1, \dots, i_k\}$ such that $\mu_i({\bf x}^*)>0$, then  such a point is not a solution for \eqref{without-equlities-998}.

2.3. If $\mu_{i}({\bf x}^*)\leq 0$, for all $i\in \{i_1, \dots, i_k\}$, and at least one is zero, then we cannot decide if ${\bf x}^*$ is a solution of \eqref{without-equlities-998}. In the literature, this corresponds to the fact that the strictly complementary conditions are not satisfied (see \cite{nocedal}, Definition 12.5, pp. 321). \hfill $\Box$\medskip

In many real-world applications, the computational burden makes the use of numerical methods essential. The algorithm presented above enables the application of steepest descent and Newton-type optimization algorithms.

We have shown that for solving the initial problem \eqref{without-equlities-998} it is sufficient to proceed stratum by stratum. Once we fix a stratum $\Pi {\bf S}_{(i_1,\dots,i_k)}$, the problem splits in two parts. First, on $\mathcal{O}_{(i_1,\dots, i_k)}$ solve numerically, using embedded steepest descent as in \cite{first-order} or embedded Newton method as in \cite{thomson} and \cite{second-order}, the classical optimization problem with constraints 
\begin{equation*}\label{}
\begin{aligned}
\underset{{\bf x}\in \mathcal{O}_{(i_1,\dots, i_k)}}{\text{argmin}}\quad & f_0(\mathbf{x}) \\
\text{s.t.}\quad & f_{i_1}({\bf x})=0, \dots,  f_{i_k}({\bf x})=0. 
\end{aligned}
\end{equation*}
Second, choose the solutions ${\bf x}^*$ that also satisfy the necessary conditions $f_j({\bf x}^*)<0,\,j\notin\{i_1,\dots,i_k\}$ and  $\mu_{i}({\bf x}^*)\leq 0$, for all $i\in \{i_1, \dots, i_k\}$. If, moreover, $\mu_{i}({\bf x}^*)<0$ for all $i\in \{i_1, \dots, i_k\}$, we have obtained a solution for \eqref{without-equlities-998}.

\section{Double ice-cream cone example}

Let ${\bf c}=(c_1,c_2,c_3)$, where all components of ${\bf c}$ are nonzero real numbers.  Consider the following optimization problem:
\begin{equation*}\label{}
\begin{aligned}
\underset{{\bf x}\in \R^3}{\text{argmin}}\quad & f_0(\mathbf{x}):=c_1x_1+c_2x_2+c_3x_3 \\
\text{s.t.}\quad & f_1({\bf x}):=x_1^2+x_2^2+x_3^2-1 \le 0, \\
 &  f_2({\bf x}):=x_1^2+x_2^2-x_3^2 \le 0.
\end{aligned}
\end{equation*}
Adding the slack variables $z_1$ and $z_2$, we transform the above optimization problem with inequality constraints into a classical optimization problem with equality constraints:  
\begin{equation*}\label{}
\begin{aligned}
\underset{({\bf x}, {\bf z})\in \R^3\times \R^2}{\text{argmin}}\quad & F_0(\mathbf{x}, {\bf z}) \\
\text{s.t.}\quad &  F_1({\bf x}, {\bf z}):=x_1^2+x_2^2+x_3^2 +\frac{1}{2}z_1^2-1 = 0, \\
 & F_2({\bf x}, {\bf z}):=x_1^2+x_2^2-x_3^2 +\frac{1}{2}z_2^2= 0,
\end{aligned}
\end{equation*}
where $F_0({\bf x}, {\bf z}):=f_0({\bf x})$.\medskip

{\bf I.} First we study the structure of the set ${\bf F}^{-1}({\bf 0}_2)$. The Jacobian of the constraint functions is 
$$\begin{bmatrix}
2x_1 & 2x_1 \\
2x_2 & 2x_2 \\
2x_3 & -2x_3 \\
z_1 & 0 \\
0 & z_2
\end{bmatrix}^T.$$
In order for ${\bf 0}_2\in \R^2$ to be a regular value for ${\bf F}$, we need to exclude the two points $(0,0,0,\pm\sqrt{2},0)$ from ${\bf F}^{-1}({\bf 0}_2)$.
The manifold  ${\bf F}^{-1}({\bf 0}_2)^{\text{reg}}={\bf F}^{-1}({\bf 0}_2)\backslash \{(0,0,0,\pm\sqrt{2},0)\}$ has the following partition:
\begin{align*}
& {\bf S}_0   =\left\{({\bf x}, {\bf z})\in{\bf F}^{-1}({\bf 0}_{2}) \,|\,z_1\neq 0, z_2\neq 0\right\}\subset {\bf F}^{-1}({\bf 0}_2)^{\text{reg}}, \\
 & {\bf S}_{(1)}   =\left\{({\bf x}, {\bf z})\in{\bf F}^{-1}({\bf 0}_2)^{\text{reg}}\,|\,z_1=0, z_2\neq 0\right\}, \\
 & {\bf S}_{(2)}   =\left\{({\bf x}, {\bf z})\in{\bf F}^{-1}({\bf 0}_2)^{\text{reg}} \,|\,z_1\neq 0, z_2= 0\right\}, \\
 &  {\bf S}_{(1,2)}   =\left\{({\bf x}, {\bf z})\in{\bf F}^{-1}({\bf 0}_2)^{\text{reg}} \,|\,z_1=0, z_2= 0\right\}
\end{align*}
and its projection is 
\begin{align*}
 &\Pi{\bf S}_0   =\left\{{\bf x}\in \R^3 \,|\, x_1^2+x_2^2+x_3^2-1 < 0, x_1^2+x_2^2-x_3^2 < 0\right\}, \\
  &\Pi {\bf S}_{(1)}   =\left\{{\bf x}\in \R^3 \,|\, x_1^2+x_2^2+x_3^2-1 = 0, x_1^2+x_2^2-x_3^2 < 0\right\}, \\
 &\Pi {\bf S}_{(2)}  = \left\{{\bf x}\in \R^3\backslash \{(0,0,0)\} \,|\, x_1^2+x_2^2+x_3^2-1 < 0, x_1^2+x_2^2-x_3^2 = 0\right\}, \\
 & \Pi {\bf S}_{(1,2)}   =\left\{{\bf x}\in \R^3 \,|\, x_1^2+x_2^2+x_3^2-1 = 0, x_1^2+x_2^2-x_3^2 = 0\right\}.
\end{align*}
The geometrical description of the above sets is as follows (see Figure \ref{con2}):
\begin{enumerate}[(i)]
\item $\Pi{\bf S}_0$ is the intersection between the open ball and the open solid cone;
\item $\Pi{\bf S}_{(1)}$ is formed by the two open spherical caps; 
\item $\Pi{\bf S}_{(2)}$ is formed by the skin of the cone that is inside the open ball minus the origin;
\item $\Pi {\bf S}_{(1,2)}$ is formed by the two circles given by the intersection of the sphere with the skin of the cone.
\end{enumerate}

{\bf II.} It is clear that on $\Pi{\bf S}_0$ we find no solution for the optimization problem since ${\bf 0}\neq {\bf c}=\nabla f_0$.\medskip

{\bf III.} We proceed to solve the optimization problem on each stratum.\bigskip

$\bullet$ {\bf Search for solutions on $\Pi {\bf S}_{(1)}$ (the case when $f_1$ is the active constraint).} \medskip

{\bf Step 1.} In this case $\mathcal{O}_{(1)}=\R^3\backslash \{(0,0,0)\}$ and the Lagrange multiplier function is $\mu_1({\bf x})=\frac{\left<{\bf c}, {\bf x}\right>}{2\|{\bf x}\|^2}$. The system \eqref{sistem-de-rezolvat} becomes
\begin{equation*}
\nabla f_0({\bf x})-\mu_{1}({\bf x})  \nabla f_{1}({\bf x})={\bf 0}.
\end{equation*}
Using the constraint $\|{\bf x}\|^2=1$, it is equivalent with
$$
{\bf c}-\left<{\bf c},{\bf x}\right>{\bf x}={\bf 0}.
$$
The solutions of this system are ${\bf x}^*_{\pm}=\pm\frac{\bf c}{\|{\bf c}\|}$, which verify the active constraint. These points are candidate solutions for the optimization problem if and only if 
$
c_1^2+c_2^2-c_3^2<0.
$
\medskip

{\bf Step 2.} Under the condition $c_1^2+c_2^2-c_3^2<0$, we have $\mu_1({\bf x}^*_{\pm})=\pm\frac{\|{\bf c}\|}{2}$. 

Because $\mu_1({\bf x}^*_{+})=\frac{\|{\bf c}\|}{2}>0$, the candidate ${\bf x}^*_+$ cannot be a solution for our argmin problem (actually is a solution for the argmax problem).

The only remaining candidate is ${\bf x}^*_{-}$. It becomes a solution if condition \eqref{conditie-Hessiana-22} is verified. In our case  \eqref{conditie-Hessiana-22} it is clearly verified, since
$$\left(\|{\bf c}\|\mathbb{I}_3\right)_{|T_{{\bf x}^*_{-}}\Pi {\bf S}_{(1)}\times T_{{\bf x}^*_{-}}\Pi {\bf S}_{(1)}}>0.$$

$\bullet$ {\bf Search for solutions on $\Pi {\bf S}_{(2)}$ (the case when $f_2$ is the active constraint).}\medskip 

{\bf Step 1.} In this case $\mathcal{O}_{(2)}=\R^3\backslash \{(0,0,0)\}$ and the Lagrange multiplier function is $\mu_2({\bf x})=\frac{c_1x_1+c_2x_2-c_3x_3}{2\|{\bf x}\|^2}$. The system \eqref{sistem-de-rezolvat} becomes
\begin{equation*}
\nabla f_0({\bf x})-\mu_{2}({\bf x})  \nabla f_{2}({\bf x})={\bf 0}.
\end{equation*}
or equivalently:
$$\begin{cases}
-c_{2} x_1x_{2}+c_{3} x_1 x_{3}+c_{1} x_{2}^{2}+c_1 x_{3}^{2}=0 \\
-c_{1} x_{1} x_{2}+c_{2} x_{1}^{2}+c_{2} x_{3}^{2}+c_{3} x_{2} x_{3}=0 \\
x_3(c_{1} x_{1}+c_{2} x_{2}+c_{3} x_{3})=0,
\end{cases}$$
which has the one parameter set of solutions $\alpha(c_1,c_2,-c_3)$, with $\alpha\in \R^*$. The candidate solutions are the ones the verify $c_1^2+c_2^2+c_3^2<\frac{1}{\alpha^2}$ and $c_1^2+c_2^2-c_3^2=0$.

When  $c_1^2+c_2^2-c_3^2\neq 0$, then we have no candidate solutions.

If $c_1^2+c_2^2-c_3^2=0$ and by hypothesis $c_3\neq 0$, we have the following family of candidate solutions $\alpha(c_1,c_2,-c_3)$, with $\alpha\in  \left(-\frac{\sqrt{2}}{2|c_3|},0\right)\cup \left(0, \frac{\sqrt{2}}{2|c_3|}\right)$.\medskip

{\bf Step 2.} Under the condition $c_1^2+c_2^2-c_3^2=0$ and $\alpha\in  \left(-\frac{\sqrt{2}}{2|c_3|},0\right)\cup \left(0, \frac{\sqrt{2}}{2|c_3|}\right)$, each point ${\bf x}^*_{\alpha}:=\alpha(c_1,c_2,-c_3)$ is a candidate solution and $\mu_2({\bf x}^*_{\alpha})=\frac{1}{2\alpha}$. 

When $\alpha\in \left(0, \frac{\sqrt{2}}{2|c_3|}\right)$, we obtain $ \mu_2({\bf x}^*_{\alpha})>0$ and consequently the candidate ${\bf x}^*_{\alpha}$ cannot be a solution for our argmin problem.

When $\alpha\in  \left(-\frac{\sqrt{2}}{2|c_3|},0\right)$, we obtain $ \mu_2({\bf x}^*_{\alpha})<0$ and so, it is necessary to check the positive definiteness of the Hessian matrix 
$$-\frac{1}{\alpha} \begin{bmatrix}
1 & 0 & 0 \\
0 & 1 & 0 \\
0 & 0 & -1\end{bmatrix}_{|T_{{\bf x}^*_{\alpha}}\Pi {\bf S}_{(2)}\times T_{{\bf x}^*_{\alpha}}\Pi {\bf S}_{(2)}}.
$$ 
Choosing the basis ${\bf e}_1:=(c_3,0,-c_1)$ and ${\bf e}_2:=(0, c_3, -c_2)$ for the tangent space $T_{{\bf x}^*_{\alpha}}\Pi {\bf S}_{(2)}$, the Hessian matrix becomes 
$$-\frac{1}{\alpha} \begin{bmatrix}
c_2^2 & -c_1c_2  \\
-c_1c_2 & c_1^2 \end{bmatrix}.$$ 
This matrix has as eigenvalues 0 and $-\frac{1}{\alpha}c_3^2$, thus being positive semi-definite, but not positive definite. We cannot decide if ${\bf x}^*_{\alpha}$ are solutions for our initial argmin problem.\bigskip

$\bullet$ {\bf Search for solutions on $\Pi {\bf S}_{(1,2)}$ (the case when both $f_1$ and $f_2$ are active constraints). }

{\bf Step 1.} In this case $\mathcal{O}_{(1,2)}=\R^3\backslash (\{(x_1,x_2,0)\,|\,x_1,x_2\in \R\}\cup\{(0,0,x_3)\,|\,x_3\in \R\})$ and the Lagrange multiplier functions (after using the two active constraints) are
$$\mu_1({\bf x})=\frac{c_1x_1+c_2x_2+c_3x_3}{2},\,\,\mu_2({\bf x})=\frac{c_1x_1+c_2x_2-c_3x_3}{2}.$$
 The system \eqref{sistem-de-rezolvat} becomes
\begin{equation*}
\nabla f_0({\bf x})-\mu_{1}({\bf x})  \nabla f_{1}({\bf x})-\mu_{2}({\bf x})  \nabla f_{2}({\bf x})={\bf 0}.
\end{equation*}
or equivalently:
$$\begin{cases}
x_{2} \left(c_{1} x_{2}-c_{2} x_{1}\right)=0 \\
x_{1} \left(c_{1} x_{2}-c_{2} x_{1}\right)=0.
\end{cases}$$
The candidate solutions on $\Pi {\bf S}_{(1,2)}$ are the four points
$${\bf u}_{\pm}^*=\left(\frac{c_1}{\sqrt{2(c_1^2+c_2^2)}}, \frac{c_2}{\sqrt{2(c_1^2+c_2^2)}}, \pm \frac{\sqrt{2}}{2}\right),\,\,
{\bf v}_{\pm}^*=\left(-\frac{c_1}{\sqrt{2(c_1^2+c_2^2)}}, -\frac{c_2}{\sqrt{2(c_1^2+c_2^2)}}, \pm \frac{\sqrt{2}}{2}\right).$$

{\bf Step 2.} The Lagrange multipliers for the above four candidate solutions are
$$\mu_1({\bf u}^*_{\pm})=\frac{\sqrt{c_1^2+c_2^2}-c_3}{2\sqrt{2}},\,\,\mu_2({\bf u}^*_{\pm})=\frac{\sqrt{c_1^2+c_2^2}+c_3}{2\sqrt{2}}$$
and 
$$\mu_1({\bf v}^*_{\pm})=-\frac{\sqrt{c_1^2+c_2^2}-c_3}{2\sqrt{2}},\,\,\mu_2({\bf v}^*_{\pm})=-\frac{\sqrt{c_1^2+c_2^2}+c_3}{2\sqrt{2}}.$$

For the points ${\bf u}^*_{\pm}$, at least one of the Lagrange multipliers is strictly positive. Consequently, these two points cannot be solutions for our initial argmin problem.

Under the condition $c_1^2+c_2^2>c_3^2$, for both points ${\bf v}^*_{\pm}$ the Lagrange multipliers are strictly negative.  It is necessary to check the positive definiteness of the Hessian matrices
$$\text{Hess} \left({f_{0}}_{|\Pi{\bf S}_{(1,2)}}\right)({\bf v}^*_{\pm})= \left[\begin{array}{ccc}
\sqrt{2 c_{1}^{2}+2 c_{2}^{2}} & 0 & 0 
\\
 0 & \sqrt{2 c_{1}^{2}+2 c_{2}^{2}} & 0 
\\
 0 & 0 & \mp c_{3} \sqrt{2} 
\end{array}\right]_{|T_{{\bf v}^*_{\pm}}\Pi {\bf S}_{(1,2)}\times T_{{\bf v}^*_{\pm}}\Pi {\bf S}_{(1,2)}}.
$$ 
Choosing the basis ${\bf e}:=(-c_2,c_1,0)$ for the tangent spaces $T_{{\bf v}^*_{\pm}}\Pi {\bf S}_{(1,2)}$, the Hessian matrices become 
$$ \begin{bmatrix}
\sqrt{2}(c_1^2 +c_2^2)^{\frac{3}{2}}  \end{bmatrix}.$$ 
Since the above expression is strictly positive, we conclude that both ${\bf v}^*_+$ and ${\bf v}^*_-$ are local minima and therefore are solutions for the initial argmin problem. 

Under the condition $c_1^2+c_2^2<c_3^2$ one of the Lagrange multipliers is strictly positive and consequently, ${\bf v}^*_+$ and ${\bf v}^*_-$ cannot be solutions for our initial argmin problem.

Under the condition $c_1^2+c_2^2=c_3^2$ one of the Lagrange multipliers is zero, and the other one is strictly negative and the Hessian matrices $\text{Hess} \left({f_{0}}_{|\Pi{\bf S}_{(1,2)}}\right)({\bf v}^*_{\pm})$ are positive definite, consequently we cannot decide if ${\bf v}^*_+$ and ${\bf v}^*_-$ are solutions for our initial argmin problem.\medskip




It follows that for the ''cone case'' $c_1^2 + c_2^2 = c_3^2$, the first- and second-order tests are inconclusive. We give an ad hoc analysis by directly inspecting the values of the objective function $f_0$. Since the double ice-cream cone is a compact set and $f_0$ is a continuous function, when restricted to this set $f_0$ has global extrema. From the previous analysis the candidates are $\mathbf{v}_{\pm}^*$, $\mathbf{x}^*_{\alpha}$ with
$\alpha\in\!\left(-\tfrac{\sqrt{2}}{2|c_3|}, 0\right)$, and possibly the origin of the cone. 

By direct computation we obtain:
\begin{align*}
&f_0({\bf u}^*_-)=\frac{\sqrt{2}}{2}(|c_3|-c_3);\,\,\,
f_0({\bf u}^*_+)=\frac{\sqrt{2}}{2}(|c_3|+c_3);\\
&f_0({\bf v}^*_-)=-\frac{\sqrt{2}}{2}(|c_3|+c_3);\,\,\,
f_0({\bf v}^*_+)=\frac{\sqrt{2}}{2}(-|c_3|+c_3);\\
&f_0({\bf 0})=f_0({\bf x}^*_{\alpha})=0,~~~\forall \alpha\in  \left(-\frac{\sqrt{2}}{2|c_3|},0\right)\cup \left(0, \frac{\sqrt{2}}{2|c_3|}\right). 
\end{align*}

Therefore, it follows that for $c_3>0$ the (global) minimizer of the objective function $f_0$ is ${\bf v}^*_-$ (and the (global) maximizer is ${\bf u}^*_+$), while for $c_3<0$ the (global) minimizer of the objective function $f_0$ is ${\bf v}^*_+$ (and the (global) maximizer is ${\bf u}^*_-$).

In order to illustrate this case (see Figure \ref{con2}), we consider ${\bf c}=(1,1,\sqrt{2})$. The point ${\bf v}_-^*$ is the global minimum and $f_0({\bf v}_-^*)=-2$. The continuous family $\mathbf{x}^*_{\alpha}$ with
$\alpha\in\!\left(-\tfrac{\sqrt{2}}{2|c_3|}, 0\right)$ and ${\bf v}_+^*$ are local minima, where the value of $f_0$ in these points is zero. The vertex of the cone behaves like a ''saddle'' point. Analogously,  the point ${\bf u}_+^*$ is the global maximum and $f_0({\bf u}_+^*)=2$, the continuous family $\mathbf{x}^*_{\alpha}$ with
$\alpha\in\!\left(0,\tfrac{\sqrt{2}}{2|c_3|}\right)$ and ${\bf u}_-^*$ are local maxima, where the value of $f_0$ in these points is also zero.


\begin{figure}
\centering
\includegraphics[scale=0.23]{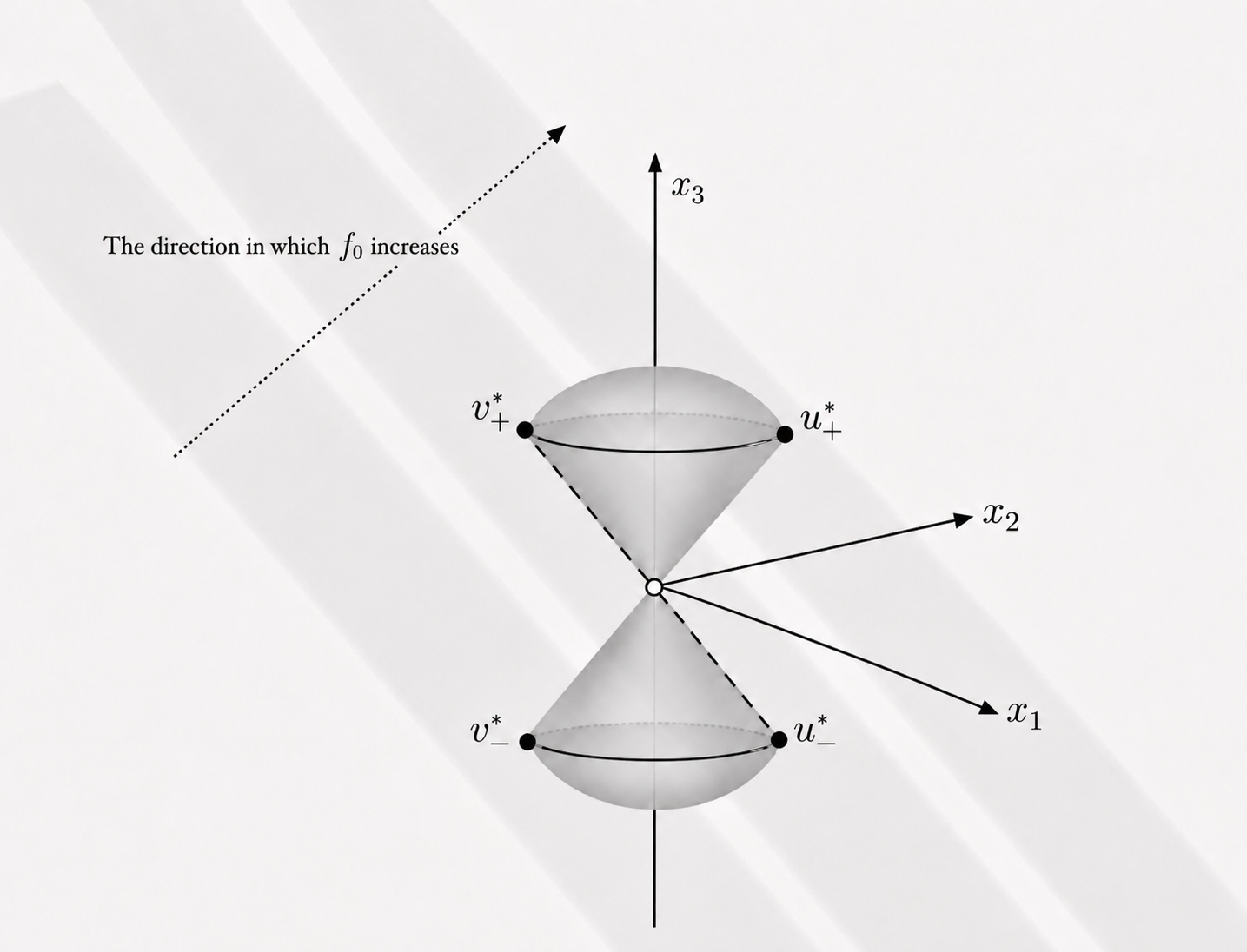}
\caption{The planes represent the level sets of $f_0$. The point ${\bf v}_-^*$ is the global minimum. The continuous family $\mathbf{x}^*_{\alpha}$ with
$\alpha\in\!\left(-\tfrac{\sqrt{2}}{2|c_3|}, 0\right)$ and ${\bf v}_+^*$ are local minima.}
\label{con2}
\end{figure}

\bigskip

{\bf Summary.} We have seen that the solutions of the argmin problem are dependent on the relative position of the vector ${\bf c}$ to the cone. The following table summarizes the discussion. \medskip

\begin{table}[h!]
\centering
\renewcommand{\arraystretch}{2.2}
\begin{tabular}{|>{\bfseries}m{2.2cm}|m{2.5cm}|m{9.5cm}|}
\hline
\textbf{Position of $\mathbf{c}$} & \textbf{Condition} & \textbf{Solutions of the argmin problem} \\
\hline
Interior of the cone
& $c_1^2 + c_2^2 < c_3^2$
& The only solution of the argmin problem (i.e. local minima) is $\mathbf{x}^*_{-}=-\dfrac{\mathbf{c}}{\|\mathbf{c}\|}$, found on the stratum $\Pi\mathbf{S}_{(1)}$. \\
\hline
Exterior of the cone
& $c_1^2 + c_2^2 > c_3^2$
& Two solutions $\mathbf{v}^*_+$ and $\mathbf{v}^*_-$ of the argmin problem, both found on the stratum $\Pi\mathbf{S}_{(1,2)}$. \\
\hline
On the cone
& $c_1^2 + c_2^2 = c_3^2$
& \begin{minipage}[t]{9.5cm}
The first order method produces the candidates
$\mathbf{u}_{\pm}^*,\, \mathbf{v}_{\pm}^* \in \Pi\mathbf{S}_{(1,2)}$
and two continuous families
$\mathbf{x}^*_{\alpha}\in \Pi\mathbf{S}_{(2)}$,
$\alpha\in\!\left(-\tfrac{\sqrt{2}}{2|c_3|},0\right)\!\cup\!\left(0,\tfrac{\sqrt{2}}{2|c_3|}\right)$.

\medskip
The second order method excludes $\mathbf{u}_{\pm}^*$
and $\mathbf{x}^*_{\alpha}$ with
$\alpha\in\!\left(0,\tfrac{\sqrt{2}}{2|c_3|}\right)$
and cannot decide the nature of the remaining candidates.\\ A direct inspection of the level sets of the objective function $f_0$ shows that for $c_3>0$ the (global) minimizer is ${\bf v}^*_-$ and for $c_3<0$ the (global) minimizer is ${\bf v}^*_+$.
\vspace{2pt}
\end{minipage} \\
\hline
\end{tabular}
\caption{Solutions of the argmin problem depending on the position of $\mathbf{c}$ relative to the cone.}
\end{table}


\end{document}